\documentclass[12pt,reqno]{amsart}
\usepackage{hyperref}
\usepackage{soul}
\usepackage{ amssymb }
\numberwithin{equation}{section}
\usepackage{hyperref}
\usepackage{float}
\usepackage{pgf,tikz}

\newcommand \reg{\operatorname{reg}}

\newcommand \Tor{\operatorname{Tor}}

\newcommand \Vr{\operatorname{Vert}}

\newcommand \cl{\operatorname{cl}}

\newcommand \St{\text{St}}
\newcommand \K{\mathbb{K}}

\newtheorem{theorem}{Theorem}[section]

\newtheorem{question}[theorem]{Question}
\newtheorem{remark}[theorem]{Remark}

\newtheorem{conj}[theorem]{Conjecture}
\newtheorem{corollary}[theorem]{Corollary}

\newtheorem*{notation*}{Notation}

\begin{document}

\author{Arindam Banerjee}
\address{Ramakrishna Mission Vivekananda Educational and Research Institute, Belur, India}
\email{123.arindam@gmail.com}

\author{Eran Nevo}

\address{Einstein Institute of Mathematics,
The Hebrew University of Jerusalem.}
\email{nevo@math.huji.ac.il}


\title{Regularity of Edge Ideals Via Suspension}

\begin{abstract}
We study the Castelnuovo-Mumford regularity of powers of edge ideals. We prove that if $G$ is a bipartite graph, then $\reg(I(G)^s) \leq 2s+\reg I(G)-2$ for all $s\geq 2$, which is the best possible upper bound for any $s$. Suspension plays a key role in proof of the base case $s=2$.
\end{abstract}
\maketitle

\section{Introduction}
Let $M$ be a finitely generated graded module over a polynomial ring $R = \K[x_1,\ldots,x_n]$, where
$\K$ is a field. The Castelnuovo-Mumford regularity (or simply, regularity) $\reg(M)$ of $M$
is defined as
$$
 \reg(M)=\max \{j-i \mid \Tor_i^R(M, \K)_j \neq 0\}.
$$
 $\text{  }$ $\text{  }$Regularity  is an important invariant in commutative algebra and algebraic geometry that measures the complexity of ideals, modules, and sheaves. A question that has been studied by many is how the regularity behaves with respect to taking powers of homogenous ideals. It is known that in the long-run $\reg(I^k)$ is linear in $k$, that is, there exist integers $a(I), b(I), c(I)$ such that $\reg(I^k) = a(I)k + b(I)$ for all $k\geq c(I)$ (see \cite{CHT,Ko}). For various classes of ideals people have studied these integers and also have looked for various upper and lower bounds for $\reg(I^k)$. 
 For monomial ideals these invariants as well as bounds reflect the underlying combinatorics (see e.g. \cite{ABH,HVT,HHZ,MS,MV,N} for various works under this theme).
 For monomial ideals $I$ generated in same degree $d$, Kodiyalam~\cite{Ko} showed that $a(I)=d$.  

 One important class of monomial ideal is the class of edge ideals $I(G)$ of finite simple graphs, namely the ideals generated by squarefree monomials of degree two. For edge ideals, $c(I(G))\le 2$ for various cases: for example when the underlying graph is either cochordal or gap and cricket free or bipartite with $\reg(I(G))\leq 3$
 (see \cite{AA,AB,ABH,Fro}). 
 As of $b(I(G))$, all known examples have
 \begin{equation}\label{eq:b(I)}
b(I(G))\leq \reg(I(G))-2
 \end{equation}
 and it is conjectured (see e.g. \cite{ABH,JS}) that this inequality holds for any graph. For various classes of graphs (e.g cochordal)
 we have $b(I(G))=\reg(I(G))-2$ so this upper bound is tight if holds.
Our Theorem~\ref{thm:main}(ii) below verifies inequality (\ref{eq:b(I)}) for all $G$ bipartite.
 Clearly this bound this sharp, for example if we take any complete bipartite graph with nonempty edge set then $\reg(I(G)^s)=2s$ for all $s$ by Theorem 2.6 below.


 Our main theorem is the following:
\begin{theorem}\label{thm:main}
(i) Let $G$ be a finite simple graph. Then:
$$ \text{ }\reg(I(G)^2) \leq \reg I(G)+2.$$
(ii) Further, if $G$ is also bipartite, then for all $s\geq 2$ we have: $$\text{ } \reg(I(G)^s) \leq 2s+\reg I(G)-2.$$
\end{theorem}

Part (i) is proved topologically, via Hochster's formula and various uses of Mayer-Vietoris long exact sequence. Part (ii) for $s>2$ is proved algebraically, via various uses of short exact sequences for related ideals. Our part two improves the main result of ~\cite{JNS}, which proves that if $G$ is bipartite, then for all $s\geq 2$  the $\text{ } \reg(I(G)^s) \leq 2s+\text{Cochord(G)}-1$, where Cochord(G) is the cochordal number of $G$ (see \cite{JNS} for definition). Their bound is known to be not sharp, whereas our bound is.

\textbf{Outline}: preliminaries are given in Section~2, Theorem~\ref{thm:main} is proved in Section~3, and concluding remarks are given in Section~4.

\section{Preliminaries}\label{pre}
  In this section, we set up the basic definitions and notation needed for the main results. Let $M$ be a finitely generated graded $R=\K[x_1,\ldots,x_n]$-module.
Write the graded minimal free resolution of $M$ in of the form:
\[
  0 \longrightarrow \bigoplus_{j \in \mathbb{Z}} R(-j)^{\beta_{p,j}(M)}
\overset{\psi_{p}}{\longrightarrow} \cdots \overset{\psi_1}{\longrightarrow}
\bigoplus_{j \in \mathbb{Z}} R(-j)^{\beta_{0,j}(M)}
\overset{\psi_0}{\longrightarrow} M\longrightarrow 0,
 \]
 where $p \leq n$, $R(-j)$ indicates the ring $R$ with the
shifted grading such that, for all $a \in \mathbb{Z}$, $R(-j)_{a}=R_{a-j}$.
The non-negative integers $\beta_{(i,j)}(M)$ are called
$i^{\text{th}}$-graded Betti number of $M$ in degree $j$.

The Castelnuovo-Mumford regularity (or regularity) of $M$ is defined to be
\[
 \reg(M)=\max\{j-i \mid \beta_{i,j}(M)\neq 0\}.\label{def_reg}
\]
  Let $I$ be a nonzero proper homogeneous ideal of $R$.
Then it follows from the definition
that $\reg(R/I) = \reg(I)-1.$

\text{ }  Let $I$ be any ideal of $R$ and $a\in R$ any element, the the \emph{colon ideal} $(I:a)$ is defined by $(I:a):=(b \text{ } | b\in R, ab\in I)$.

    \emph{Polarization} is a process that creates a squarefree monomial out of a monomial, possibly in a larger ring. If $f=x_1^{e_1}\ldots x_n^{e_n}$ is a monomial in $\mathbb{K}[x_1, \ldots , x_n]$ then polarization of $f$ is defined as $\tilde{f}=x_{11}\ldots x_{1e_1} x_{21}\ldots x_{2e_2}\ldots x_{n1} \ldots x_{n e_n}$ in the ring $\mathbb{K}[x_{11},\ldots x_{1e_1}, x_{21}\ldots x_{2e_2}, \ldots, \\ x_{n1} \ldots, x_{ne_n}]$.
    For convenience
    we identify the variable $x_{i1}$ with $x_i$, so the new polynomial ring extends the old one.
    For a monomial ideal $I$ with minimal monomial generators $\{m_1,\ldots, m_k\}$, we define the \emph{polarization} of $I$ as $\tilde{I}:=(\tilde{m_1}, \ldots, \tilde{m_k})$ in a suitable ring, see e.g
   \cite{HHZ} or \cite[Sec.1.6]{K1}. In the special case where degree of a variable $u=x_i$ is two in some generator we call the unique new variable $x_{i1}$ a \emph{whisker variable} and denote it by $u'$ for short.
  In this paper we repeatedly use one of the important properties of the
polarization:
\begin{theorem}(e.g.~\cite[Cor.1.6.3(a)]{K1}) Let $I$ be a monomial
ideal in $R$.
Then
$$\reg(I)=\reg(\widetilde{I}).$$
\end{theorem}

One of the main technique that is used in this paper is that of short exact sequences. In particular we shall use the following well known result~\cite[Lem.2.11]{AB}:

\begin{theorem}
(i) Let $I$ be a homogeneous ideal in a polynomial ring $R$ and $m$ be an element of degree $d$ in $I$. Then the following is a short exact sequence: $$0 \longrightarrow \frac{R}{(I:m)} \overset{.m} \longrightarrow \frac{R}{I} \longrightarrow \frac{R}{I+(m)} \longrightarrow 0$$
Hence: $$\reg(I) \leq \max \{\reg (I:m) +d, \reg (I+(m))\}.$$
(ii) In case $I$ is square free and $x$ a variable, then also $\reg(I,x) \leq \reg I$.
\end{theorem}

 Let $G$ be a finite simple graph with $V(G)=\{x_1\ldots.x_n\}$ the varibles of $R$. The the edge ideal $I(G)$ of $G$ is defined as the ideal in $R$: $$I(G)=(x_i x_j|x_ix_j \in E(G)).$$

 For example, edge ideal of a $5$-cycle is $(x_1x_2,x_2x_3,x_3x_4,x_4x_5,x_5x_1)$.

 The next couple of theorems allow for induction when increasing the power of an edge ideal.

\begin{theorem}(\cite[Thm.5.2]{AB})
For any simple graph $G$ and any $s \geq 1,$ let the set of minimal monomial generators of $I(G)^s$ be $\{m_1, \dots, m_k\}.$ Then

$$ \reg I(G)^{s+1} \leq \max \{\reg I(G)^s , \reg (I(G)^{s+1}: m_l)+2s , 1 \leq l \leq k\}.$$
\end{theorem}

 By identifying the variables with the vertices of $G$, interpreting edges as square free quadratic monomials, defining neighborhood for any vertex $c$,  $N(c):=\{z\in V(G): cz \in E(G)\}$ and using  \cite[Thm.5.2]{AB}  we get the following corollary:

\begin{corollary}\label{cor:I^2:e}
(i) The ideal $(I(G)^{s+1}: m_l)$ is a quadratic monomial ideal, and,

(ii) for the special case where $s=1$ and $m=ab$ is an edge then $$(I(G)^2:ab)=I(G)+(xy| x\in N(a),y\in N(b)).$$
\end{corollary}

 For bipartite graphs we further have:

\begin{theorem}(\cite[Prop.3.5, Lem.3.7]{AA})
For any finite simple bipartite graph $G$ we have $(I(G)^{s+1}:e_1\cdots e_s)=((I(G)^{k+1}:e_1\cdots e_k)^{s-k} : e_{k+1}\cdots e_s)$ where for all $i$ we have $e_i \in E(G)$.

Further,
$(I(G)^{k+1}:e_1\cdots e_k)$ is an edge ideal of a bipartite graph on same bipartite partition of vertices as $G$.
\end{theorem}

 Now we recall some basic definitions about graphs and simplicial complex that will be useful.

Let $G$ be a finite simple graph with vertex set $V(G)$
and edge set $E(G)$.  A subgraph $H \subseteq G$  is called \textit{induced} if $\{u,v\}$ is an edge of $H$ if and only if $u$ and $v$ are vertices of $H$ and $\{u,v\}$ is an edge of $G$.

For $u\in V(G)$, let $N_G(u) = \{v \in V (G)\mid \{u, v\} \in E(G)$ and $N_G[u]=N_G(u)\cup \{u\}$.
For $U \subseteq V(G)$, denote by $G \setminus U$
the induced subgraph of $G$ on the vertex set $V(G) \setminus U$.

Let $G$ be a graph. We say $2$ disjoint edges
  $\{f_1,f_2\}$ form an $2K_2$ in $G$ if $G$ does not have an edge with one
endpoint in $f_1$ and the other in $f_2$. A graph without $2K_2$ is called $2K_2$-free or \textit{gap-free} graph. The \textit{complement} of a graph $G$, denoted by $G^c$, is the graph on the same
vertex set in which $\{u,v\}$ is an edge of $G^c$ if and only if it is not an edge of $G$.
Then  $G$ is gap-free if and only if $G^c$ contains no induced $4$-cycle.

A graph $G$ is chordal (also called triangulated) if every induced
cycle in $G$ has length $3$, and is co-chordal if the complement graph $G^c$ is chordal.
 The following important theorem(s) characterizes the edge ideals with regularity 2, and the regularity of their powers.


\begin{theorem}
(i) Fr\"oberg~(\cite[Thm.1]{Fro}): For any finite simple graph $G$, we have $G^c$ is chordal if and only if $\reg(I(G))=2$. 

(ii)Herzog-Hibi-Zheng~(\cite[Thm.1.2]{HHZh}) Further, in this case $\reg(I(G)^s)=2s$ for any $s$.
\end{theorem}

  A \emph{simplicial complex} $\Delta$ on a vertex set $\{1,\ldots, n\}$ is a collection of subsets of $\{1,\ldots ,n\}$ such that if $\tau \in \Delta, \sigma \subseteq \tau$ then we have $\sigma \in \Delta$. The \emph{induced subcomplex} $\Delta [A]$ of $\Delta$ on vertex set $A\subset \{1,\ldots,n\}$ is the collection of faces $\tau'$ of $\Delta$ such that $\tau' \subset A$. Clearly the induced subcomplex is a simplicial complex itself. We denote by $\Vr(\Delta)$ the set of vertices of $\Delta$.

 The \emph{link} of a face $d$ in$ \Delta$ is: $$\text{link}_d (\Delta)= \{ \tau |\tau \cup \{d\} \in \Delta, \{d\} \cap \tau =\emptyset\}$$

 The (open) \emph{star} of a face $\sigma$ in a simplicial complex $\Delta$ is the set of all faces that contain $\sigma$, namely $\St_d(\Delta)=\{\tau| \tau \in \Delta, d \in \tau\}$. The \emph{closed star} $\bar{\St}_d (\Delta$ ) of $d$ is defined by the smallest subcomplex that contains $\St_d(\Delta)$. The \emph{antistar} is defined as the subcomplex $\text{ast}_d (\Delta) = \{\tau \in \Delta| \tau \cap \{d\} = \emptyset\}$.



 The \emph{join} of two simplicial complexes $\Delta_1$ and $\Delta_2$ is defined by $\Delta_1 * \Delta_2=\{\sigma \cup \tau| \sigma \in \Delta_1, \tau \in \Delta_2\}$. The \emph{suspension} of a simplicial complex $\Delta$, w.r.t. two points $a$ and $b$, is the join defined by $\Sigma_{a,b} \Delta=\Delta * \{\{a\},\{b\},\{\emptyset\}\}$; its geometric realization is homeomorphic to the topological suspension of the space $\Delta$.

 Let $G$ be a finite simple graph and $\Delta=\cl(G^c)$, where $\cl(G)$ denotes the corresponding clique complex of a graph $G$.

The following formulation of regularity follows from the so called Hochster's formula
(See {\cite{MS}}for further details):

\begin{theorem}[Hochster's formula]
For any finite simple graph $G$ whose edge set is nonempty and $\Delta=\cl(G^c)$ we have:
$$\reg(G):=\reg(I(G))= \max\{l+2:\exists W \subseteq \Vr[\Delta], \tilde{H}_l(\Delta[W];k)\neq 0\}.$$
\end{theorem}

\section{Main Results}
Let $G$ be a finite simple graph, $ab \in E(G)$ and $G'=G \cup \{xy: x\neq y,  ax,by \in E(G)\}$. Let $\Delta=\cl(G^c)$ and $\Delta'=\cl(G'^c)$.
 We first prove $\reg(I(G)^2:e) \leq \reg(I(G))$ for every edge $e$ of $G$. This will lead us to our main result via a series of short exact sequence arguments. For that we first prove:

\begin{theorem}\label{thm:reg(D')}
 $\reg (\Delta ' ) \leq \reg(\Delta).$
\end{theorem}

\begin{proof}

Let $A=\Vr(\bar{\St}_a (\Delta)),B=\Vr(\bar{\St}_b(\Delta)), C=A\cap B$ and $D=\Vr(\Delta)-(A\cup B)$. With this notation we observe the following:$$ \Delta'=(\Delta[A]\cup_{\Delta[C]} \Delta[B])\cup_{\Delta[C]} (\Delta[C] \cup_{d\in D} (\{d\}*\Delta[C \cap \Vr(\St_d (\Delta))]))$$

 We get this equality simply from the definition of $G'$, where every neighbour of $a$ is connected to every neighbour of $b$.\\

 Let $reg(\Delta') = l+2$ and $W$ be such that $\tilde{H}_l (\Delta'[W]) \neq 0$. Assume by contradiction
  $\reg(\Delta) < l+2$.

 Decompose $W=W_1 \cup_{W_c} W_2 \subseteq V$ where $W_1=W\cap(A\cup_C B), W_2=W\cap(C\cup D)$ and $W_C=W_1\cap W_2=C \cap W$. With these we have the following:

\textbf{Claim One:} $\tilde{H}_l \Delta'[W_1] =0$.

\text{Proof of Claim:} Consider the following Mayer-Vietoris long exact sequence:$$\tilde{H}_l \Delta'[W_1\cap A] \bigoplus \tilde{H}_l \Delta'[W_1\cap B] \longrightarrow \tilde{H}_l \Delta'[W_1] \longrightarrow \tilde{H}_{l-1} \Delta'[W_C]$$ If the first and the third terms of this sequence are zero then so will be the middle/second term by exactness. By taking suspension we note that $\tilde{H}_l (\Sigma_{a,b}( \Delta'[W_c]))=\tilde{H}_l (\Delta[\{a,b\}\cup W_C])$, which vanishes by the assumption on $\Delta$.  Thus we also have  $\tilde{H}_{l-1} \Delta'[W_C]=0$. Hence the third term is zero. Also the first term is zero, by assumption, because $\Delta'[W_1\cap A]= \Delta[W_1\cap A]$ and $\Delta'[W_1\cap B]= \Delta [W_1\cap B]$ by definition.
Done.

\textbf{Claim Two:} $\tilde{H}_l \Delta'[W_2] =0$.

\text{Proof of Claim:} We prove  by induction on $|D \cap W_2|$. If the size of this intersection is zero then $W_2 \subseteq C$ and the claim follows by assumption on $\Delta$ as $\Delta'[C]=\Delta[C]$.

Let $D'$ be defined by $W_2 \cap D =D' \cup \{d\}$ and denote $X=\Delta'[W_2]$. We prove the claim using the Mayer-Vietoris long exact sequence
corresponding to the union
$X=\text{ast}_d X\cup_{\text{link}_d X} \bar{\text{St}}_d X$:


$$\tilde{H}_l \text{ast}_{d} X \bigoplus \tilde{H}_l \overline{\text{St}}_{d} X \longrightarrow \tilde{H}_l X \longrightarrow \tilde{H}_{l-1} \text{link}_d X$$

Like before we prove that the end terms are zero.
As $\text{ast}_{d}X=\Delta'[W_2\setminus\{d\}]$,
by induction $\tilde{H}_l \text{ast}_{d}X$ is zero. Clearly $\tilde{H}_l \overline{\text{St}}_{d} X$ is zero as the space is a cone. Thus the first term is zero. For the last term we note that by definition $\text{link}_d X= (\text{link}_d \Delta) \cap \Delta[W_2 \cap C]$. Now since $\Delta$ is clique complex of a graph  we have  $(\text{link}_d \Delta) \cap \Delta[W_2 \cap C]=\Delta[\Vr(\text{link}_d \Delta) \cap W_2 \cap C]$. Let us define $\tilde{W}:=\Vr(\text{link}_d \Delta) \cap W_2 \cap C$. 



 By assumption $\tilde{H}_l \Delta [\{a,b\} \cup \tilde{W}]=0$. Hence $\tilde{H}_{l-1} \Delta [ \tilde{W}]=0$.

 Recall that $\Delta'[W]=\Delta'[W_1] \cup_{\Delta'[C]} \Delta'[W_2]$. Consider the following Mayer-Vietoris exact sequence: $$\tilde{H}_l \Delta'[W_1] \bigoplus \tilde{H}_l \Delta'[W_2] \longrightarrow \tilde{H}_l \Delta'[W] \longrightarrow \tilde{H}_{l-1} (\Delta'[W_C])$$

  By previous two claims the left term is zero. As $\tilde{H}_l \Delta'[W]\neq 0$ we have $\tilde{H}_{l-1} (\Delta'[W_C])\neq 0$. But as $\Delta'[W_C]=\Delta[W_C]$ we have
  $0 \neq \tilde{H}_l(\Delta[W_C] * \{a,b\})= \tilde{H}_l\Delta[\{a,b\} \cup W_C]$,  a contradiction to the assumption.
  Thus $\reg (\Delta) \geq l+2$.
\end{proof}

We are now ready to prove Theorem~\ref{thm:main}.


\begin{proof}[Proof of Theorem~\ref{thm:main}]
(i) We first prove that $\reg(I(G)^2:ab) \leq \reg(I(G))$ for every edge $ab\in E(G)$. Let $J=(I(G)) +(uv|u\neq v, u \in N(a),v\in N(b))$. By Corollary~\ref{cor:I^2:e}
we have:$$(I(G)^2:ab)=J+(u^2| u \in N(a) \cap N(b)).$$
By Theorem~2.1 $\reg(I(G)^2:ab) ) = \reg(\widetilde{(I(G)^2:ab)})$, the polarization. Here $L:=\widetilde{(I(G)^2:ab)}=
J+(uu'| u \in N(a) \cap N(b))$
for new whisker variables $u'$
 in a larger polynomial ring (defined in Sec.2).
So enough to prove that $\reg(L) \leq \reg(I)$.

Now let $N(a) \cap N(b)=\{u_1, \ldots, u_k\}$. Consider the following short exact sequences:
$$0 \longrightarrow \frac{R}{(L:u_1)} (-1) \rightarrow \frac{R}{L} \rightarrow \frac{R}{(L,u_1)} \rightarrow 0$$
$$0 \longrightarrow \frac{R}{((L, u_1):u_2)} (-1) \rightarrow \frac{R}{(L, u_1))} \rightarrow \frac{R}{(L,u_1,u_2)} \rightarrow 0$$
$$\vdots$$
$$0 \longrightarrow \frac{R}{((L,u_1,\ldots u_{n-1}):u_n)} (-1) \rightarrow \frac{R}{(L, u_1, \ldots , u_{n-1})} \rightarrow \frac{R}{(L,u_1,\ldots ,u_n)} \rightarrow 0$$
  Now observe that $(L,u_1,\ldots ,u_n)=J+(\text{variables})$ and for every $i$,  $(L,u_1,\ldots u_{i-1}):(u_i)=(L:u_i)+ (\text{variables})$. By repeated use of Theorem 2.2 (both parts) we have that $\reg (L) \leq \max \{ \reg(L:u)+1, u\in N(a) \cap N(b), \reg (J)\}$. Now by Theorem~\ref{thm:reg(D')} $\reg(J) \leq \reg(I)$. Enough to show that $ \reg(L:u)+1 \leq \reg(I)$ for all $u \in N(a) \cap N(b)$.

  Now we have $(L:u)= (J:u)+(u')$, so $\reg(L:u)=\reg(J:u)$.
  Let $J':=I(G[V\setminus N(u)])$, then $(J:u)=J'+(\text{variables})$ so $\reg(J:u)\le \reg(J')$ as $J'$ is a square-free monomial ideal. Note that
  $J'+(ab)=I(G[V\setminus (N(u)-\{a,b\})]$, thus $\reg(J'+(ab))\le \reg I$.
  As the ideals $J'$ and $(ab)$ are defined over disjoint sets of variables, the regularity of their sum is the sum of regularities minus $1$ (see e.g.~\cite[Lem.8]{W}, namely
  $\reg(J'+(ab))=\reg J'+\reg (ab)-1=\reg(J')+1$.
  Putting together, we get
  $\reg(L:u)\le \reg(J') \leq \reg(I) -1$. This finishes the proof of part (i).

(ii) In part (i) we have already proved that $\reg(I(G)^2:e)\leq \reg(I(G))$. Hence by Theorem 2.5
we have for any generator $e_1\ldots e_{s-1}$ of $I(G)^{s-1}$, where $e_1,\ldots,e_{s-1}$ are (not necessarily) different edges, the following:
$$ \reg(I(G)^s:e_1\ldots e_{s-1})=\reg((I(G)^{s-1}:e_2\ldots e_{s-1})^2:e_1) \leq \reg(I(G)^{s-1}: e_2\ldots e_{s-1}).$$
As this is true for any $s$, proceeding inductively we get $\reg(I(G)^s:e_1\ldots e_{s-1}) \leq \reg(I(G))$. Hence by Theorem 2.3 and induction we get the result.
\end{proof}

\section{Further Research}
 In this section we discuss some questions for further research. Our main result immediately leads to the following question whose answer has been conjectured to be positive by various people (see~\cite[Conj.A]{ABH} and~\cite[Conj.1.1]{JS}):

\begin{question}
For any finite simple graph $G$, is it true that $\reg I(G)^s \leq 2s+ \reg(I(G)) -2$ for all $s$?
\end{question}

  Due to the asymptotic stability we have that for an homogeneous ideal $I$ generated in degree $d$ we have an integer $c(I)$ such that $\reg(I^{s+1})-\reg(I^s)=d$ for all $s\geq c(I)$. We have proved that for all bipartite graphs $G$ we have $\reg(I(G)^s)- \reg(I(G)) \leq 2s-2$. However the behaviour of the sequence $\{\reg(I^s)\}$ can be irregular for smaller $s$ values even for edge ideals. In fact there are examples of bipartite graphs where $\reg(I(G)^2)=\reg(I(G)) +1$ (for example one can check that this is the case for the bipartite edge ideal $(x_1y_1,x_2y_2,x_3y_3,x_4y_4,x_1y_2,x_2y_4,x_3y_1,x_4y_3)$).
  
  Can $c(I)$ be bounded by some simple invariants of $I$, for homogenous ideals?
  Conca~\cite{C} showed that for any given integer $d > 1$ there exists an ideal $J$ generated by $d+5$ monomials of degree $d + 1$ in $4$ variables such that $\reg(J^k) = k(d + 1)$ for every $k <d$ and $\reg(J^d) \geq d(d + 1) + d-1$. In particular, $c(I)$ cannot be bounded above in terms of the number of variables only, not even for monomial ideals in general. Further, a result of Raicu~\cite{Raicu} gives binomial ideals $I_n$ on $n^2$ variables, generated in degree $2$, with $c(I_n)=n-1$. Thus, the following question arise:

\begin{question}
For homogeneous ideals $I$ on $n$ variables, generated in degree $d$, is $c(I)$ bounded above by a function of $d$ and $n$?
\end{question}

 It has been conjectured by Banerjee and Mukundan~\cite{AV}
 that for all bipartite graphs $G$, we have $c(I(G)) \leq 2$. It is known  for cochordal, gap free plus cricket/diamond/$4$-cycle free~\cite{AB,HHZh,Er1,Er2}. Apart from edge ideals, it was shown by Conca and Herzog~\cite{CH} that polymatroidal ideals have linear resolutions and powers of polymatroidal ideals are polymatroidal ideals. So for the class of all polymatroidal ideals $c(I)=1$.

Finally we conclude by a discussion on a related conjecture by~\cite{NP}:

\begin{conj}([23])\label{conj:Nevo-Peeva}
If $\reg(I(G)) \leq 3$ and $G^c$ has no induced $4$-cycle then for all $s \geq 2$ $\reg (I(G))^s =2s$.
\end{conj}
 Theorem 2.3 was proved by Banerjee in his thesis to study this conjecture and related other problems, based on simple Theorem~2.2.
 We now explain why this inductive approach via colon ideals can \emph{not} be used directly to settle Conjecture~\ref{conj:Nevo-Peeva}.

Any $2$-dimensional simplicial complex $\Delta$ can be subdivided so that the resulted complex is flag-no-square, see~\cite{Dr} or~\cite[Lem.2.3]{PS}, i.e. $\Delta=\text{cl}(H)$ where $H$ is a graph with no \emph{induced} $4$-cycles. In particular,
we choose such $H$ so that $\Delta$ triangulates the dunce hat, a contractible $2$-dimensional complex.
Thus, all subcomplexes of $D$ have vanishing homology in dimension $\ge 2$. Further, the link of any vertex $a\in \Delta$ is an induced subcomplex (as $\Delta$ is a clique complex) with nonzero first homology. For an edge $ab\in G:=H^c$, the construction of $\Delta'$ from  Theorem~\ref{thm:reg(D')} satisfies $\text{link}_a \Delta'=\text{link}_a \Delta$ is an induced subcomplex of $\Delta'$.

We conclude that $\reg(I(G))=3$ and by Corollary~\ref{cor:I^2:e}(ii)
for any edge $ab\in G$ also $\reg((I(G)^2:ab)=3$.
Thus, if $\reg(I(G)^2)=4$ as Conjecture~\ref{conj:Nevo-Peeva} suggests, then Theorem 2.2 can not be directly applied to prove it.

On the other hand, if $\reg(I(G)^2)>4$ then this will be a counter example. Unfortunately we could not verify the value of $\reg(I(G)^2)$ due to computational limitations. It will be great if this can be verified in future.

\noindent \textbf{Acknowledgements} 
We thank Aldo Conca for pointing us to~\cite{Raicu}. 

Most of this work was done when first author was visiting the Institute of Mathematics of Hebrew University and he would like to thank faculty and staff of Hebrew University for their hospitality. Also, first author was partially supported by DST INSPIRE (India) research grant (DST/INSPIRE/04/2017/000752) and he would like to acknowledge that.

Second author was partially supported by Israel Science Foundation grant ISF-1695/15, by grant 2528/16 of the ISF-NRF Singapore joint research program, and by ISF-BSF joint grant 2016288.



\end{document}